\documentclass{amsart}
\usepackage {amsmath, amscd}
\usepackage {amssymb}
\usepackage{hyperref}

\def\<{\langle}
\def\>{\rangle}

\numberwithin{equation}{section}

\def\K{{\bf K}}

\def\T{{\bf T}}

\def\BB{{\mathcal B}}

\def\DD{{\mathcal D}}

\def\HH{{\mathcal H}}

\def\KK{{\mathcal K}}

\def\MM{{\mathcal M}}

\def\PP{{\mathcal P}}

\def\RR{{\mathcal R}}

\def\XX{{\mathcal X}}
\def\YY{{\mathcal Y}}

\def\bbC{\mathbb{C}}
\def\bbZ{\mathbb{Z}}

\def\bbD{\mathbb{D}}
\def\bbT{\mathbb{T}}

\def\HHH{\mathfrak{H}}
\def\LLL{\mathfrak{L}}
\def\lll{\mathfrak{l}}

\newcommand{\Comp}{C}

\newtheorem{lemma}{Lemma}[section]
\newtheorem{theorem}[lemma]{Theorem}
\newtheorem{corollary}[lemma]{Corollary}

\theoremstyle{definition}

\newtheorem{exer}[lemma]{Exercise}

\let\oldtocsection=\tocsection

\let\oldtocsubsection=\tocsubsection

\renewcommand{\tocsection}[2]{\hspace{0em}\oldtocsection{#1}{#2}}
\renewcommand{\tocsubsection}[2]{\hspace{1em}\oldtocsubsection{#1}{#2}}

\newenvironment{exercise}%
{\parindent=1cm\bigskip\begin{minipage}{11cm}\begin{exer}\begin{small}}%
{\end{small}\end{exer}\end{minipage}\bigskip}

\title{A short introduction to de Branges--Rovnyak spaces}

\author{Dan Timotin}

\address{Institute of Mathematics of the Romanian Academy, P.O. Box 1-764, Bucharest 
014700, Romania}
\email{Dan.Timotin@imar.ro}

\thanks{The author is partially supported by a grant of the Romanian National Authority for Scientific
Research, CNCS Ð UEFISCDI, project number PN-II-ID-PCE-2011-3-0119.}

\begin{document}

\maketitle

\tableofcontents
\begin{abstract}The notes provide a short introduction to de Branges--Rovnyak spaces. They cover some basic facts and are intended to give the reader a taste of the theory, providing sufficient motivation to make it interesting.
\end{abstract}

\section{Introduction}

The purpose of these notes is to provide a short introduction to de Branges--Rovnyak spaces, that have been introduced in~\cite{BR, BR1}, an area that has seen significant research activity in the last years. They are  intended to give to a casual reader a taste of the theory, providing sufficient motivation and connections with other domains to make it, hopefully, interesting. There exist  two comprehensive references on the subject: the older book of Sarason~\cite{S} that contains most of the basic facts, and the more recent monograph of Fricain and Mashreghi~\cite{FM}. The interested reader may study in depth the subject from there.  

The prerequisites are basic facts in operator theory on Hilbert space and in Hardy spaces. Many books contain them, but we have preferred to give as a comprehensive reference Nikolski's treaty~\cite{N}, where all can be found (see the beginning of Section~\ref{se:prelim}).

So the plan of the notes is the following. We  introduce the de Branges--Rovnyak spaces as natural reproducing kernel spaces, generalizing the model spaces that appear prominently in the theory of contractions. The challenge here is to justify in a sufficient manner a rather exotic object of study, namely contractively included subspaces;  we have considered the reproducing kernel approach as especially convenient. We give next some basic properties, following closely~\cite{S}. 

The dichotomy $b$ extreme/nonextreme appears soon. The general idea is that the extreme case has many features that are not far from the case of $b$ inner (the classical model spaces), while the nonextreme case is more exotic from this point of view. 

Originally, the de Branges--Rovnyak spaces have been developed in view of model theory, that is, giving a universal model for certain class of operators on Hilbert space. They have been in the shade for a few decades, as the much more popular and well developed model theory  of Sz--Nagy and Foias~\cite{NF} has gained the upper hand. It is  known to experts that  the two theories are equivalent, and we thought that a justification of the study of de Branges--Rovnyak spaces should include some presentation of their role as model spaces. It turns out that this is easier done for the extreme case, and we have chosen to present this case at the end of the notes.

As noted above, the basic reference for de Branges--Rovnyak spaces, that has been frequently used  in these notes, is the book of Sarason~\cite{S}. Some simple unproved results appear in the text as exercises; for the others, references  are indicated in the text at the relevant places, occasionally with a hint of the proof.

\section{Preliminaries}\label{se:prelim}

A comprehensive reference for all facts in this section is~\cite{N}, which has the advantage to contain both the necessary prerequisites from function theory (Part A, chapters 1--3) and from operator theory (part C, chapter 1).

If $H$ is a Hilbert space and $H'\subset H$ a closed subspace, we will write $P_{H'}$ for the orthogonal projection onto $H'$. The space of bounded linear operators from $H_1$ to $H_2$ is denoted $\BB(H_1, H_2)$; in case $H_1=H_2=H$ we write just $\BB(H)$. If $T\in\BB(H_1, H_2)$ is a contraction, we will denote by $D_T$ the selfadjoint operator $(I-T^*T)^{1/2}$ and by $\DD_T$ the closed subspace $\overline{D_T H_1}=\ker T^\perp$. Thus $\DD_T\subset H_1$ and $\DD_{T^*}\subset H_2$. Obviously $D_T|\DD_T$ is one-to-one.

\begin{exercise}\label{ex:defects}
We have $TD_T=D_{T^*}T$. In particular, $T$ maps $\DD_T$ into $\DD_{T^*}$.
\end{exercise}

We may also consider the domain and/or the range of $D_T$ to be $\DD_T$; by an abuse of notation all these operators will  still be denoted  by $D_T$. Note that the adjoint of $D_T:\DD_T\to H_1$ is $D_T:H_1\to \DD_T$.

\begin{exercise}\label{exr:defect of restriction}
If $T\in \BB(H)$ is a contraction, $H_1\subset H$ is a closed subset invariant by $T$, and we denote $T_1=T|H_1$, then $\DD_{T_1}=\overline{P_{H_1}\DD_T}$.
\end{exercise}

We denote by $L^2, L^\infty$  the Lebesgue spaces on the unit circle $\bbT$; we will also  meet their closed subspaces  $H^2\subset L^2$  and $H^\infty\subset L^\infty$ (the \emph{Hardy spaces}). The corresponding norms will be denoted by $\|\cdot\|_2$ and $\|\cdot\|_\infty$ respectively. As usual, $H^2$ and $H^\infty$ can be identified with their analytic extension inside the unit disc $\bbD$. We assume known basic facts about inner and outer functions. We will write $P_+:=P_{H^2}$ (the orthogonal projection in $L^2$). The symbols $\bf 0$ and $\bf 1$ will denote the constant functions that take this value.

Each function $\phi$ in $L^\infty$ acts as multiplication on $L^2$; the corresponding operator will be denoted by $M_\phi$, and we have $\|M_\phi\|=\|\phi\|_\infty$. In particular, if $\phi(z)=z$, we will write $Z=M_\phi$. Actually, the commutant of $Z$ (the class of all operators $T$ on $L^2$ such that $ZT=TZ$) coincides precisely with the class of all $M_\phi$ for $\phi\in L^\infty$. (Obviously we have to define $\phi=T{\bf 1}\in L^2$; a little work is needed to show that it is in $L^\infty$.)

The compression $P_+M_\phi P_+$ restricted to the space $H^2$ is called the Toeplitz operator with symbol $\phi$ and is denoted by $T_\phi$. Again we have $\|T_\phi\|=\|\phi\|_\infty$; moreover, if $\phi\in H^\infty$, then $T_\phi$ is one-to-one (this is a consequence of the brothers Riesz Theorem: a function in $H^2$ is $\not=0$ a.e.).  In particular, if $\phi(z)=z$, we will write $S=T_\phi$; its adjoint $S^*$ acts as
\begin{equation}\label{eq:action of S*}
(S^*f)(z)=\frac{f(z)-f(0)}{z}.
\end{equation} 
We have $T_\phi^*=T_{\bar \phi}$. 

As noted above, the multiplication operators commute; this is in general not true for the Toeplitz operators. 

\begin{exercise}\label{exr:multiplication of toeplitz}
If $\phi\in H^\infty$, or $\psi\in \overline{H^\infty}$, then $T_\psi T_\phi=T_{\psi\phi}$.
\end{exercise}

If $k_\lambda(z)=\frac{1}{1-\bar\lambda z}$ (a \emph{reproducing vector} in $H^2$---see more  on reproducing kernels in Subsection~\ref{sse:rep kernels} below), then for any $\phi\in H^\infty$ we have
\begin{equation}\label{eq:toeplitz acting on reproducing kernels}
T_{\bar \phi}k_\lambda=\overline{\phi(\lambda)}k_\lambda.
\end{equation}

\section{Introducing de Branges--Rovnyak spaces}

\subsection{Model spaces}\label{sse:model spaces}

Beurling's theorem says that any subspace of $H^2$ invariant by $S$ is of the form $u H^2$, with $u$ an inner function. 

\begin{exercise}\label{exr:s|thetaH2}
$S|u H^2$ is unitarily equivalent to $S$.
\end{exercise}

From some points of view, the orthogonal $\K_u=H^2\ominus u H^2$ is more interesting: it is a \emph{model space}. It is invariant by $S^*$, but $S^*|\K_u$ may behave very differently. Actually, we know exactly \emph{how differently}:

\begin{theorem}\label{th:unitary equivalence for inner function}
If $T$ is a contraction on a Hilbert space $H$, then the following are equivalent:
\begin{enumerate}
\item
$I-T^*T$ and $I-TT^*$ have rank one and $T^n$ tends strongly to 0.

\item
$T$ is unitarily equivalent to $S_u:=S^*|\K_u\in\BB(\K_u)$ for some inner function $u$.
\end{enumerate}
\end{theorem}

Theorem~\ref{th:unitary equivalence for inner function} is a particular case of the general Sz.-Nagy--Foias theory of contractions (see, for instance, the revised edition~\cite{NF}  of the original monography); one can find it also in~\cite{N}.

So $S_u$ is a \emph{model operator} for a certain class of contractions. We will meet in this course  model operators for a more general class.

\subsection{Reproducing kernels}\label{sse:rep kernels}

We  introduce a larger class of spaces that include $\K_u$ for inner $u$; this will be done by means of reproducing kernels. A Hilbert space $\RR$ of functions on a set $X$ is called a \emph{reproducing kernel space} (RKS) if the evaluations at points of $X$ are continuous; we will always have $X=\bbD$. By Riesz's representation theorem it follows then that for each $\lambda\in\bbD$ there exists a function $\lll ^\RR_\lambda\in\RR$, called the \emph{reproducing vector} for $\lambda$, such that $f(\lambda)= \< f, \lll ^\RR_\lambda \>$. The function of two variables $\LLL^\RR(z,\lambda):=\lll ^\RR_\lambda(z)= \< \lll ^\RR_\lambda, \lll ^\RR_z \>$ is called the \emph{reproducing kernel} of the space $\RR$. There is a one-to-one correspondence between RKS's and positive definite kernels (see for instance~\cite{Ar}).

\begin{exercise}\label{exr:reproducing kernel for subspace}
\begin{enumerate}
\item
If $\RR$ is a RKS, and $\RR_1\subset \RR$ is a closed subspace, then $\RR_1$ is also a RKS, and $\lll ^{\RR_1}_\lambda=P_{\RR_1}\lll ^\RR_\lambda$. 
\item
If $\RR=\RR_1\oplus \RR_2$, then
\begin{equation}\label{eq:reproducing kernels for orthogonal sums}
\LLL^\RR=\LLL^{\RR_1}\oplus \LLL^{\RR_2}.
\end{equation}
\end{enumerate}
\end{exercise}

All three  spaces discussed above have reproducing kernels, namely:
\[
\begin{matrix}
H^2& \longrightarrow &\frac{1}{1-\bar\lambda z},\medskip
\\
u H^2& \longrightarrow &\frac{\overline{u(\lambda)}u(z)}{1-\bar\lambda z},\medskip\\
\K_u &\longrightarrow &\frac{1-\overline{u(\lambda)}u(z)}{1-\bar\lambda z},
\end{matrix}
\]
and one can check that equality~\eqref{eq:reproducing kernels for orthogonal sums} is satisfied. 
\smallskip

Our plan is to obtain RKSs with similar formulas, but replacing the \emph{inner} function $u$ with an \emph{arbitrary} function $b$ in the unit ball of $H^\infty$. That is, we want spaces with kernels $\frac{\overline{b(\lambda)}b(z)}{1-\bar\lambda z}$ and $\frac{1-\overline{b(\lambda)}b(z)}{1-\bar\lambda z}$.

\smallskip

Of course it is not obvious that such RKSs exist. Then, if they exist, we want to identify them concretely, hoping to relate them to the familiar space~$H^2$.

Things are simpler for the first kernel. Note first the  next (general) exercise.

\begin{exercise}\label{exr:modifying kernel by scaling}
If $\LLL(z,\lambda)$ is a positive kernel on $X\times X$ and $\phi:X\to\bbC$, then $\phi(z)\overline{\phi(\lambda)}\LLL(z,\lambda)$ is a positive kernel.  
\end{exercise}

So $\frac{\overline{b(\lambda)}b(z)}{1-\bar\lambda z}$ is the kernel of \emph{some} space $\RR$; but we would like to know it more concretely. The case $b=u$ inner suggests that a good candidate might be $bH^2$. Now, we already have a problem: if $b$ is a general function, $bH^2$ might not be closed in $H^2$, so it is not a genuine Hilbert space. But let us be brave and go on:  we want 
\[
\lll ^\RR_\lambda(z)=\overline{b(\lambda)}b(z)k_\lambda.
\]
Since the reproducing kernel property should be valid in $\RR$, we must have, for any $f\in H^2$,
\[
b(\lambda)f(\lambda)= \< bf,\lll ^\RR_\lambda(z) \>_\RR 
=b(\lambda) \< bf,b k_\lambda \>_\RR
\] 
and therefore 
\[
f(\lambda)=\< bf,b k_\lambda \>_\RR.
\]
On the other hand, since $f\in H^2$, we have $f(\lambda)= \< f,k_z \>_{H^2}$. 

We have now arrived at the crucial point. If $b$ is inner, then $\< bf,b k_z \>_{H^2}= \< f,k_z \>_{H^2}$ and everything is fine: the scalar product in $bH^2$ is the usual scalar product in $H^2$. But, in the general case, we have to define a \emph{different} scalar product on $\RR=bH^2$, by the formula
\begin{equation}\label{eq:scalar product on bH2}
\< bf,bg \>_{\RR}:= \< f,g \>_{H^2}.
\end{equation}

This appears to solve the problem. Since $bf_1=bf_2$ implies $f_1=f_2$, formula~\eqref{eq:scalar product on bH2} is easily shown to define a scalar product on $bH^2$. We will denote the corresponding Hilbert space by $\MM(b)$. Let us summarize the results obtained.

\begin{theorem}\label{th:about M(b)}
$\MM(b)$, defined as $bH^2$ endowed with the scalar product~\eqref{eq:scalar product on bH2}, is a Hilbert space, which as a set is a linear subspace (in general not closed) of $H^2$. Its reproducing kernel is $\frac{\overline{b(\lambda)}b(z)}{1-\bar\lambda z}$, and the inclusion $\iota:\MM(b)\to H^2$ is a contraction.
$\MM(b)$ is invariant by $S$, and $S$ acts as an isometry on $\MM(b)$.
\end{theorem}

\begin{proof}
From~\eqref{eq:scalar product on bH2} it follows that the map $f\mapsto bf$ is isometric from $H^2$ onto $\MM(b)$, whence $\MM(b)$ is complete. 
 The formula for the reproducing kernel has been proved (in fact, it lead to the definition of the space $\MM(b)$). We have 
\[
\|\iota(bf)\|_{H^2}=\|bf\|_{H^2}\le \|f\|_{H^2}=\|bf\|_{\MM(b)},
\]
and thus $\iota$ is a contraction.

Finally $\MM(b)$ is invariant by $S$ since $z(bf)=b(zf)$, and
\[
\|zbf\|_{\MM(b)}=\|b(zf)\|_{\MM(b)}=\|zf\|_{H^2}=\|f\|_{H^2}= \|bf\|_{\MM(b)},
\]
and thus the restriction of $S$ is an isometry.
\end{proof}

This settles the case of the kernel$\frac{\overline{b(\lambda)}b(z)}{1-\bar\lambda z}$. To discuss  $\frac{1-\overline{b(\lambda)}b(z)}{1-\bar\lambda z}$ is slightly more complicated, and a preliminary discussion is needed.

\subsection{Contractively included subspaces}
Let  $T:E\to H$ be a bounded one-to-one operator. Define on the image $T(E)$ a scalar product $ \< \cdot,\cdot \>'$ by the formula
\begin{equation}\label{eq:pre-scalar product}
\< T\xi,T\eta \>' := \< \xi,\eta \>_E.
\end{equation}
Then $T$ is a unitary operator from $E$ to $T(E)$ endowed with $ \< \cdot,\cdot \>'$. The space obtained is complete; we will denote it by $M(T)$. The linear space $M(T)$ is contained as a set in $H$, and the inclusion is a contraction if and only if $T$ is a contraction. In this case  the space $M(T)$ will be called a \emph{contractively included subspace} of $H$, and the scalar product will be denoted $ \< \cdot,\cdot \>_{M(T)}$.

A slight modification is needed in case $T$ is not one-to-one; then~\eqref{eq:pre-scalar product} cannot be used directly since there $T\xi$ does not determine $\xi$ uniquely. We may however recapture this uniqueness and use~\eqref{eq:pre-scalar product} if we require that $\xi,\eta\in\ker T^\perp$; then $T$ becomes a unitary from $\ker T^\perp$ to $M(T)$.

In almost all cases  in this course the corresponding contraction $T$ will be one-to-one, and thus we will apply directly~\eqref{eq:pre-scalar product}. The only exception appears in Theorem~\ref{th:invariant to adjoints}, when we will use Lemma~\ref{le:Douglas lemma} below, with no direct reference to the scalar product.

We have already met a particular case of this notion: with the above notations, we have $\MM(b)=M(T_b)$. Remember that, since  $b\in H^\infty$, the operator $T_b$ is one-to-one.

The following is a basic result that is used when we have to deal with two contractively embedded subspaces. It is essentially contained in~\cite{D}.

\begin{lemma}\label{le:Douglas lemma}
Suppose $T_1:E_1\to H$, $T_2:E_2\to H$ are two contractions. Then:
\begin{enumerate}
\item
The space $M(T_1)$ is contained contractively in $M(T_2)$ (meaning that $M(T_1)\subset M(T_2)$ and the inclusion is a contraction) if and only if $T_1T_1^*\le T_2T_2^*$.
\item
The spaces $M(T_1)$ and $M(T_2)$ coincide as Hilbert spaces (that is, they are equal as sets, and the scalar product is the same) if and only if $T_1T_1^*= T_2T_2^*$. 
\item
$T:H\to H$ acts contractively on $M(T_1)$ (meaning that $T(M(T_1))\subset M(T_1)$ and $T|M(T_1)$ is a contraction) if and only if $TT_1T_1^*T^*\le T_1T_1^*$.
\end{enumerate}
\end{lemma}

It is worth at this point to note 
 the following theorem of de Branges and Rovnyak~\cite{BR}, which is an analogue of Beurling's theorem. 

\begin{theorem}\label{th:deBranges-Rovnyak-Beurling}
Suppose $X\subset H^2$ is a contractively included subspace of $H^2$. The following are equivalent
\begin{enumerate}
\item
$X$ is invariant by $S$ and the restriction $S|X$ is an isometry (in the norm of $X$).

\item
There exist a function $b$ in the unit ball of $H^\infty$, such that $X=\MM(b)$.
\end{enumerate}

The function $b$ is determined up to a multiplicative constant of modulus~1.
\end{theorem}

Suppose now that $H$ is a reproducing kernel Hilbert space, with kernel $\LLL(z,\lambda)$. We may obtain for the reproducing vectors of $M(T)$ a formula similar to the particular case from the previous subsection.

\begin{lemma}\label{le:repro kernel for contractively included}
Suppose $T:E\to H$ is one-to-one. With the above notations, we have
\[
\lll_\lambda^{M(T)}=TT^* \lll_\lambda^H.
\]
\end{lemma}

\begin{proof}
We have, using~\eqref{eq:pre-scalar product},
\[
\< Tf,TT^* \lll_\lambda^H \>_{M(T)} = \< f ,T^* \lll_\lambda^H \>_H=\< Tf , \lll_\lambda^H \>_H=(Tf)(\lambda)= \< Tf,\lll_\lambda^{M(T)} \>_{M(T)},
\]
which proves the theorem.
\end{proof}

In case $H=H^2$, $T=T_b$, we recapture the previous result: $M(T_b)=\MM(b)$ and, using~\eqref{eq:toeplitz acting on reproducing kernels},
\[
\lll_\lambda^{\MM(b)}=T_b T_b^* k_\lambda= b \overline{b(\lambda)} k_\lambda.
\]

\subsection{The complementary space}

Remember that our current purpose is to find, if possible, an RKS with kernel $\frac{1-\overline{b(\lambda)}b(z)}{1-\bar\lambda z}$. Let us note that
\[
\frac{1-\overline{b(\lambda)}b(z)}{1-\bar\lambda z}= \frac{1}{1-\bar\lambda z}- \frac{\overline{b(\lambda)}b(z)}{1-\bar\lambda z}=\LLL^{H^2}(z, \lambda)- \LLL^{\MM(b)}(z, \lambda).
\]
Can we obtain in the general case a formula for such a difference? The answer is positive.

\begin{lemma}\label{le:complementary space}
With the above notations,
\[
\LLL^{H}(z, \lambda)-\LLL^{M(T)}(z, \lambda)= \LLL^{M(D_{T^*})}(z, \lambda),
\]
where $D_{T^*}:\DD_{T^*}\to H$.
\end{lemma}

\begin{proof}
Using Lemma~\ref{le:repro kernel for contractively included}, we have
\begin{align*}
\LLL^{H}(z, \lambda)-\LLL^{M(T)}(z, \lambda)& = \< \lll^H_\lambda, \lll^H_z \>_{H} -  \< \lll^{M(T)}_\lambda, \lll^{M(T)}_z \>_{M(T)}\\
&= \< \lll^H_\lambda, \lll^H_z \>_{H} -  \< TT^*\lll^{H)}_\lambda,TT^* \lll^{H}_z \>_{M(T)}\\
&=\< \lll^H_\lambda, \lll^H_z \>_{H} - \< T^*\lll^{H)}_\lambda,T^* \lll^{H}_z \>_{E}\\
&= \< \lll^H_\lambda,(I-TT^*) \lll^H_z \>_{H} 
=\< D_{T^*}\lll^H_\lambda, D_{T^*} \lll^H_z \>_{H}\\
&=\< D_{T^*}^2\lll^H_\lambda, D_{T^*}^2 \lll^H_z \>_{M(D_{T^*})},
\end{align*}
(the last equality being a consequence of the fact  that $D_{T^*}$ is one-to-one as an operator from $\DD_{T^*}$ into $H$). Then Lemma~\ref{le:repro kernel for contractively included}, applied to $D_{T^*}$ instead of $T$, says that the last quantity is precisely $\LLL^{M(D_{T^*})}(z, \lambda)$.
\end{proof}

Since $D_{T^*}$ is a contraction,  the RKS corresponding to $\LLL^{H}(z, \lambda)-\LLL^{M(T)}(z, \lambda)$ is also a space contractively included in $H$; it is called the space \emph{complementary} to $M(T)$ and will be denoted by $\Comp(T)$.

\begin{exercise}\label{exr:kernels for orthogonal}
If $T$ is an isometry, then $M(T)$ is a usual subspace of $H$ (with the norm restricted), and $\Comp(T)$ is its orthogonal complement. 
\end{exercise}

If $x\in H$, then one can write
\begin{equation}\label{eq:decomposition M oplus C}
x=TT^* x+ D_{T^*}^2x.
\end{equation}

The first term in the right hand side is in $M(T)$, while the second is in $C(T)$. Moreover,
\[
\|TT^*x\|_{M(T)}^2=\|T^*x\|_H^2, \qquad \|D_{T^*}^2x\|_{C(T)}^2=\|D_{T^*}x\|_H^2,
\]
whence
\[
\|x\|^2=\|TT^*x\|_{M(T)}^2+\|D_{T^*}^2x\|_{C(T)}^2.
\]
In case $T$ is an isometry, $M(T)$ and $C(T)$ form an orthogonal decomposition of $H$, and~\eqref{eq:decomposition M oplus C} is the corresponding orthogonal decomposition of $H$. In the general case $M(T)$ and $C(T)$ may have a nonzero intersection, and so a decomposition $x=x_1+x_2$ with $x_1\in M(T)$, $x_2\in C(T)$ is no more unique.

\begin{exercise}\label{exr:uniqueness of decomposition}
If $T$ is a contraction, and $x=x_1+x_2$ with $x_1\in M(T)$, $x_2\in C(T)$, then 
\[
\|x\|^2\le \|x_1\|_{M(T)}^2+\|x_2\|_{C(T)}^2,
\]
and equality implies $x_1=TT^*x$, $x_2=D_{T^*}^2x$.
\end{exercise}

At this point we have achieved our first purpose.  Lemma~\ref{le:complementary space} applied to the case $H=K=H^2$ and $T=T_b$ yields the identification of the reproducing kernel corresponding to 
$\frac{1-\overline{b(\lambda)}b(z)}{1-\bar\lambda z}$.

\begin{theorem}\label{th:identification of de Branges}
The RKS with kernel $\frac{1-\overline{b(\lambda)}b(z)}{1-\bar\lambda z}$ is $M(D_{T^*_b})$. It is a contractively included subspace of $H^2$ that will be denoted $\HH(b)$ and called the \emph{de Branges--Rovnyak} space associated to the function $b$ in the inner ball of $H^\infty$.
\end{theorem}

As noted above, if $b=u$ is inner, then $\HH(b)=\K_u$. 

\begin{exercise}\label{exr:H(b) for b strict contraction}
\begin{enumerate}
\item
$H^2=\HH({\bf 0})=\MM({\bf1})$.
\item
If $\|b\|_\infty<1$, then $\HH(b)$ is a renormed version of $H^2$. 
\item
If  $\inf_{z\in\bbD}|b(z)|>0$, then $\MM(b)$ is a renormed version of $H^2$.
\end{enumerate}

\end{exercise}

We will denote from now on  $k_\lambda^b=\lll_\lambda^{\HH(b)}=\frac{1-\overline{b(\lambda)b(z)}}{1-\bar\lambda z}$.

\section{More about contractively included subspaces}
\begin{lemma}\label{le:H(A) si H(A^*)}
If $T:E\to H$ is a contraction, then:
\begin{enumerate}
\item
$\xi\in H$ belongs to $\Comp(T)$ if and only if $T^*\xi\in \Comp(T^*)$.
\item
If $\xi_1,\xi_2\in \Comp(T)$, then
\[
\< \xi_1, \xi_2 \>_{\Comp(T)}= \< \xi_1, \xi_2 \>_{H}+ \< T^*\xi_1, T^*\xi_2 \>_{\Comp(T^*)}.
\]
\end{enumerate}
\end{lemma}

\begin{proof}
The inclusion $T^*(\Comp(T))\subset\Comp(T^*)$ follows from the intertwining relation in Exercise~\ref{ex:defects}. On the other hand, if $T^*\xi\in \Comp(T^*)$, we have $T^*\xi=D_T\eta$ for some $\eta\in H$, and, using again Exercise~\ref{ex:defects},
\[
\xi=TT^*\xi +D_{T^*}^2\xi=TD_T\eta +D_{T^*}^2\xi =D_{T^*}( D_{T^*}\xi +T\eta),
\]
which shows that $\xi\in\Comp(T)$.

To prove (2), write $\xi_1=D_{T^*}\eta_1$, $\xi_2=D_{T^*}\eta_2$, with $\eta_1,\eta_2\in \DD_{T^*}$; then $T^*\eta_1, T^*\eta_2\in\DD_T$. Since both $D_T:\DD_T\to H$ and $D_{T^*}:\DD_{T^*}\to H$ are one-to-one, we have, using~\eqref{eq:pre-scalar product} and Exercise~\ref{ex:defects},
\[
\begin{split}
\< \xi_1,\xi_2 \>_{C(T)}&= \< \eta_1,\eta_2 \>= \< D_{T^*}\eta_1,D_{T^*}\eta_2 \> +\< T^*\eta_1,T^*\eta_2 \>\\
&= \< \xi_1,\xi_2 \>+ \<D_T T^*\eta_1,D_T T^*\eta_2 \>_{C(T^*)}\\
&=\< \xi_1,\xi_2 \>+ \< T^*D_{T^*} \eta_1, T^*D_{T^*}\eta_2 \>_{C(T^*)}\\
&= \< \xi_1, \xi_2 \>+ \< T^*\xi_1, T^*\xi_2 \>_{\Comp(T^*)}.\qedhere
\end{split}
\]
\end{proof}

There is a more direct way in which complementarity is related to orthogonality. If $T\in\BB(E,H)$ is a contraction, we define the \emph{Julia operator} $J(T):E\oplus \DD_{T^*}\to H\oplus \DD_T$ by
\[
J(T)=\begin{pmatrix}
T & D_{T^*}\\ D_T & -T^*
\end{pmatrix}.
\]

\begin{exercise}\label{exr:Julia}
The Julia operator is unitary.
\end{exercise}

\begin{lemma}\label{le:geometrical identification of M(T) and C(T)}
Suppose $T:E\to H$ is one-to-one. Denote 
\[
\XX_1=J(T)(E\oplus\{0\}),\quad \XX_2=(H\oplus \DD_T)\ominus \XX_1 =J(T)(\{0\}\oplus \DD_{T^*}),
\]
and by $P_1$ the projection of $H\oplus \DD_T$ onto its first coordinate $H$. Then $P_1|\XX_1$ is unitary  from $\XX_1$ onto $M(T)$, and  $P_1|\XX_2$ is unitary  from $\XX_2$ onto $C(T)$.
\end{lemma}

\begin{proof}
We have 
\[
P_1\XX_1=P_1(\{ Tx\oplus D_Tx: x\in E\})=\{Tx: x\in E\}=M(T).
\]
Moreover, if $x_1\in \XX_1$, then $x_1=Tx\oplus D_Tx$ for some $x\in E$, and
\[
\|P_1x_1\|_{M(T)}=\|Tx\|_{M(T)}=\|x\|_E=\|J(T)(x\oplus 0)\|=\|x_1\|,
\]
which proves the first part of the lemma.

Now $\XX_2=J(T)(\{0\}\oplus \DD_{T^*})$, so 
\[
P_1\XX_2=P_1 ( \{ D_{T^*}y\oplus -T^*y: y\in \DD_{T^*}\})=
 \{ D_{T^*}y: y\in \DD_{T^*}\}=C(T).
\]
If $x_2\in\XX_2$, then $x_2=D_{T^*}y\oplus -T^*y $ for some $y\in \DD_{T^*}$, and
\[
\|P_1x_2\|_{C(T)}=\|D_{T^*}y\|_{C(T)}=\|y\|_{\DD_{T^*}}=\|J(T)(0\oplus y)\|=\|x_2\|,
\]
as required.
\end{proof}

We can view this result as saying that the  orthogonal decomposition of $H\oplus \DD_T$ as $\XX_1\oplus \XX_2$ is mapped by projecting onto the first coordinate into the complementary decomposition $H=M(T)+C(T)$ (which is not, in general, a direct sum). So the rather exotic definition of complementary spaces is in fact the projection of a more familiar geometric structure.

\section{Back to $\HH(b)$}

\subsection{Some properties of $\HH(b)$; definition of $X_b$} 
We denote $\HH(\bar b):=\HH(T_{\bar b})$. Although our focus is on $\HH(b)$, the space $\HH(\bar b)$ is a useful tool for its study.

\begin{lemma}\label{le:H(bar b) contained in H(b)}
$\HH(\bar b)$ is contained contractively in $\HH(b)$.
\end{lemma}

\begin{proof}
We have 
\[
\begin{split}
T_bT_{\bar b}&= P_+M_bP_+ M_{\bar b}P_+|H^2\le P_+M_b M_{\bar b}P_+|H^2\\& = P_+ M_{\bar b}M_bP_+|H^2= P_+ M_{\bar b}P_+M_bP_+|H^2= T_{\bar b}T_b.
\end{split}
\]
Therefore 
\[
D_{T_b}^2\le D_{T_b^*}^2, 
\]
whence Lemma~\ref{le:Douglas lemma}(1) implies that $\HH(\bar b)$ is contained contractively in $\HH(b)$.
\end{proof}

Lemma~\ref{le:H(A) si H(A^*)} applied to the case $T=T_b$ yields the following result.
\begin{lemma}\label{le:H(b) and H( bar b)}
If $h\in H^2$, then $h\in \HH(b)$ if and only if $T_{\bar b}h\in \HH(\bar b)$. If $h_1, h_2\in \HH(b)$, then
\begin{equation}\label{eq:sc prod H(b) cu H(bar b)}
\< h_1,h_2 \>_{\HH(b)}= \< h_1, h_2 \>_{H^2}+ \< T_{\bar b}h_1,T_{\bar b}h_2 \>_{\HH(\bar b)}.
\end{equation}
\end{lemma}

We  now show  that, similarly to model spaces, de Branges--Rovnyak spaces are invariant by adjoints of Toeplitz operators.

\begin{theorem}\label{th:invariant to adjoints}
If $\phi\in H^\infty$, then $\HH(b)$ and $\HH(\bar b)$ are both invariant under $T_\phi^*=T_{\bar\phi}$, and the norm of this operator in each of these spaces is at most $\|\phi\|_\infty$.
\end{theorem}

\begin{proof}
We may assume that $\|\phi\|_\infty\le 1$. By Lemma~\ref{le:Douglas lemma}(3), in order to show that $T_{\bar \phi}$ acts as a contraction in $\HH(\bar b)$ we have to prove  the inequality
\[
T_{\bar \phi}(I-T_{\bar b}T_b)T_\phi\le I-T_{\bar b}T_b,
\]
or
\begin{align*}
0&\le I-T_{\bar b}T_b-T_{\bar \phi}(I-T_{\bar b}T_b)T_\phi=
I-T_{\bar b}T_b-T_{\bar \phi}T_\phi+T_{\bar \phi}T_{\bar b}T_bT_\phi\\
&=I-T_{|b|^2}-T_{|\phi|^2}+T_{|b|^2|\phi|^2}=
T_{(1-|b|^2)(1-|\phi|^2)}.
\end{align*}
But the last operator is the compression to $H^2$ of $M_{(1-|b|^2)(1-|\phi|^2)}$, which is positive, since $(1-|b|^2)(1-|\phi|^2)\ge 0$.

This proves the statement for $\HH(\bar b)$. Take now $h\in\HH(b)$. Lemma~\ref{le:H(b) and H( bar b)} implies that $T_{\bar b}h\in \HH(\bar b)$. By what has been just proved, $T_{\bar b}T_{\bar \phi}h =T_{\bar \phi}T_{\bar b}h\in \HH(\bar b)$, and then applying again Lemma~\ref{le:H(b) and H( bar b)} we obtain $T_{\bar b}h\in \HH(b)$.

Finally, using~\eqref{eq:sc prod H(b) cu H(bar b)} and the contractivity of $T_{\bar \phi}$ on $H^2$ as well as on $\HH(\bar b)$, we have
\[
\|T_{\bar \phi}h\|^2_{\HH(b)}= \|T_{\bar \phi}h\|^2_{H^2} + \|T_{\bar \phi}T_{\bar b}h\|^2_{\HH(\bar b)} \le  \|h\|^2_{H^2} + \|T_{\bar b}h\|^2_{\HH(\bar b)}=\|h\|_{\HH(b)}^2,
\]
so $T_{\bar \phi}$ acts as a contraction in $\HH(b)$.
\end{proof}

The most important case  is obtained when $\phi(z)=z$.  Theorem~\ref{th:invariant to adjoints} says then that $\HH(b)$ is invariant under $S^*$ and the restriction of $S^*$ is a contraction. We will denote by $X_b$ this restriction $S^*|\HH(b)$.

\begin{corollary}\label{co:S*b in h(b)}
The function $S^*b$ is in $\HH(b)$.
\end{corollary}

\begin{proof}
We have
\[
T_{\bar b}S^*b=S^*T_{\bar b}b=S^*T_{\bar b}T_b 1=-S^*(I-T_{\bar b}T_b) 1
\]
(for the last equality we have used the fact that $S^*1=0$). Obviously $(I-T_{\bar b}T_b) 1\in \HH(\bar b)$, so Theorem~\ref{th:invariant to adjoints} implies that $T_{\bar b}S^*b=-S^*(I-T_{\bar b}T_b) 1\in \HH(\bar b)$. By Lemma~\ref{le:H(b) and H( bar b)} it follows that $S^*b\in\HH(b)$.
\end{proof}

Note that if $b\not=\bf 1$, then $S^*b\not=\bf 0$.
Besides $S^*b$, we know as inhabitants of $\HH(b)$ the reproducing vectors $k_\lambda^b$. Other elements may be obtained, for instance, by applying to these elements powers or functions of $X_b$.

\begin{exercise}\label{exr:}
Show that, if $\lambda\in \bbD$, then
\[
((I-\lambda X_b)^{-1}(S^*b))(z)=\frac{b(z)-b(\lambda)}{z-\lambda}.
\]
Therefore the functions in the right hand side belong to $\HH(b)$.
\end{exercise}

In general $b$ itself may not be in $\HH(b)$; we will see later exactly when this happens.
Let us also compute the adjoint of $X_b$.

\begin{lemma}\label{le:adjoint of X}
If $h\in \HH(b)$, then 
\[
X_b^*h=Sh- \< h,S^*b \>_{\HH(b)}b.
\]
\end{lemma}

\begin{proof}
A computation shows that $X_bk^b_\lambda=S^* k^b_\lambda=\bar\lambda k^b_\lambda-\overline{b(\lambda)}S^*b$. Then, if $h\in\HH(b)$ and $\lambda\in\bbD$, then
\[
\begin{split}
(X_b^*h)(\lambda)&= \< X_b^*h,k^b_\lambda \>_{\HH(b)}=
\< h,X_bk^b_\lambda \>_{\HH(b)} =\lambda \< h, k^b_\lambda\>_{\HH(b)}
-b(\lambda) \< h,S^*b \>_{\HH(b)}\\
&=\lambda h(\lambda)- \< h,S^*b \>_{\HH(b)}b(\lambda) ,
\end{split}
\]
which proves the lemma.
\end{proof}

\subsection{Another representation of $\HH(b)$ and $X_b$}

In the sequel of the course we will use the notation $\Delta=(1-|b|^2)^{1/2}$. The spaces $\overline{\Delta H^2}$ and $\overline{\Delta L^2}$ are closed subspaces of $L^2$ invariant with respect to $Z$. We will denote by $V_{\Delta}$ and $Z_{\Delta}$ the corresponding restrictions of $Z$. 

\begin{exercise}\label{exr:Z Delta and Z Delta'}
$V_\Delta$ is isometric, while $Z_\Delta$ is unitary.
\end{exercise}

The next result, a slight modification of Lemma~\ref{le:geometrical identification of M(T) and C(T)}, provides another representation of $\HH(b)$.

\begin{theorem}\label{th:geometrical representation of de branges spaces}
Suppose that $b$ is a function in the unit ball of $H^\infty$. Then:
\begin{enumerate}
\item
$S\oplus V_\Delta$ is an isometry on $H^2\oplus \overline{\Delta  H^2}$.
\item
The space $\KK_b:=(H^2\oplus \overline{\Delta  H^2})\ominus \{bh\oplus \Delta h: h\in H^2\}$ is a subspace of $H^2\oplus \overline{\Delta  H^2}$ invariant with respect to $S^*\oplus V_\Delta^*$.
\item
The projection $P_1:H^2\oplus \overline{\Delta  H^2}\to H^2$ on the first coordinate maps $\KK_b$ unitarily onto $\HH(b)$, and $P_1(S^*\oplus V_\Delta^*)=X_bP_1$.
\end{enumerate}
\end{theorem}

\begin{proof}
The proof of (1) is immediate. Also, the map $h\mapsto bh\oplus \Delta h$ is an isometry of $H^2$ onto $\{bh\oplus \Delta h: h\in H^2\}$, which is therefore a closed subspace. Since $(S\oplus V_\Delta)(bh\oplus \Delta h)=b(zh)\oplus \Delta (zh)$, it is immediate that $\{bh\oplus \Delta h: h\in H^2\}$ is invariant by $S\oplus V_\Delta$, whence its orthogonal $\KK$ is invariant by $S^*\oplus V_\Delta^*$; thus (2) is proved.

To prove (3), let us apply Lemma~\ref{le:geometrical identification of M(T) and C(T)} to the case $T=T_b$, when $C(T)=\HH(b)$. It says that, if $\XX_2=(H^2\oplus \DD_{T_b})\ominus \{T_b h\oplus D_{T_b}h\}$, then the projection onto the first coordinate maps  $\XX_2$ unitarily onto $\HH(b)$. Since, for any $h\in H^2$,
\begin{equation*}
\|D_{T_b}h\|^2=\|h\|^2-\|bh\|^2=\frac{1}{2\pi}\int_{-\pi}^\pi|h(e^{it})|^2\, dt-  \frac{1}{2\pi}\int_{-\pi}^\pi|b(e^{it})h(e^{it})|^2\, dt= \|\Delta h \|^2,
\end{equation*}
the map $D_{T_b}h\mapsto \Delta h $ extends to a unitary $U$ from $\DD_{T_b}$ onto the closure of $\overline{\Delta  H^2}$. Then $I_{H^2}\oplus U$ maps unitarily $H^2\oplus \DD_{T_b}$ onto $H^2\oplus \overline{\Delta  H^2}$, $\XX_1$ onto $\{bh\oplus \Delta h: h\in H^2\}$, $\XX_2$ onto $\KK_b$, and it commutes with the projection on the first coordinate. Therefore $P_1$ maps $\KK_b$ unitarily onto $\HH(b)$, and 
\[
P_1(S^*\oplus V_\Delta^*)|\KK_b=
P_1(S^*\oplus V_\Delta^*)P_1|\KK_b=
(S^*\oplus 0)P_1|\KK_b=
(S^*|\HH(b))P_1=X_bP_1.\qedhere
\]
\end{proof}

\subsection{The dichotomy extreme/nonextreme}

The study of the spaces $\HH(b)$ splits further  into two mutually exclusive cases: when $b$ is an extreme point of the unit ball of $H^\infty$ and when it is not. The first case includes $b=u$ inner, and thus will be more closely related to model spaces, while the second includes the case $\|b\|<1$, and thus there will be properties similar to the whole of $H^2$. Actually, we will not use extremality directly, but rather through one of the equivalent characterizations given by the next lemma (for which again~\cite{N} can be used as a reference).

\begin{lemma}\label{le:characterization of extreme}
If $h$ is a function in the unit ball of $H^\infty$,  then the following are equivalent:
\begin{enumerate}
\item
$b$ is extreme.
\item
$\frac{1}{2\pi}\int_{-\pi}^\pi \log \Delta(e^{it})\, dt=-\infty$.
\item
$\overline{\Delta H^2}=\overline{\Delta L^2}$.
\end{enumerate}
\end{lemma}

\section{The nonextreme case}

When $\frac{1}{2\pi}\int_{-\pi}^\pi \log \Delta(e^{it})\, dt>-\infty$, there exists a uniquely defined outer function $a$ with $|a|=\Delta$ and $a(0)>0$;  thus $|a|^2+|b|^2=1$. This function is a basic tool in the theory of $\HH(b)$ in the nonextreme case. Since (see Exercise~\ref{exr:multiplication of toeplitz})
\[
T_{\bar a}T_a=T_{\bar a a}=I-T_{\bar b b}=I-T_{\bar b}T_b
\]
Lemma~\ref{le:Douglas lemma}(2) implies that $\HH(\bar b)=\MM(\bar a)$.

\begin{exercise}\label{exr:kernel of T_bar a}
If $a$ is an outer function, then $\ker T_{\bar a}=\{0\}$.
\end{exercise}

We can apply Lemma~\ref{le:H(b) and H( bar b)} to the current situation.

\begin{lemma}\label{le:H(b) and M(bar a) for b nonextreme}
\begin{enumerate}
\item
We have $h\in \HH(b)$ if and only if $T_{\bar b}h\in \MM(\bar a)$; when this happens there is a unique (by Exercise~\ref{exr:kernel of T_bar a}) function $h^+\in H^2$ such that $T_{\bar b}h=T_{\bar a}h^+$. 
\item
If $h_1, h_2\in \HH(b)$, then 
\begin{equation*}
\< h_1,h_2 \>_{\HH(b)}=\< h_1,h_2 \>_{H^2} +\< h_1^+, h_2^+ \>_{H^2}.
\end{equation*}
\item
If $h\in\HH(b)$ and $\phi\in H^\infty$, then $(T_{\bar \phi}h)^+=T_{\bar\phi}h^+$.
\end{enumerate}
\end{lemma}

\begin{proof} (1) is a consequence of
 Lemma~\ref{le:H(b) and H( bar b)} and the equality $\HH(\bar b)=\MM(\bar a)$; the uniqueness of $h^+$ follows from Exercise~\ref{exr:kernel of T_bar a}. 
 
The formula for the scalar product in Lemma~\ref{le:H(b) and H( bar b)} becomes 
\begin{equation*}
\< h_1,h_2 \>_{\HH(b)}= \< h_1,h_2 \>_{H^2} + \<T_{\bar a} h_1^+,T_{\bar a} h_2^+ \>_{\HH(\bar a)}= \< h_1^+, h_2^+ \>_{H^2},
\end{equation*}
the last equality being a consequence of the fact that $T_{\bar a}$ is one-to-one. This proves (2).

Finally, 
\[
T_{\bar b}T_{\bar \phi}h=T_{\bar \phi}T_{\bar b}h =T_{\bar \phi}T_{\bar a}h^+=T_{\bar a}T_{\bar \phi}h^+,
\]
proving (3).
\end{proof}

We gather in a theorem some properties of $\HH(b)$ for $b$ nonextreme.

\begin{theorem}\label{th:the nonextreme case}
Suppose $b$ is nonextreme. 
\begin{enumerate}
\item
The polynomials belong to $\MM(\bar a)$ and are dense in $\MM(\bar a)$.
\item
$\MM(\bar a)$ is dense in $\HH(b)$.
\item
The polynomials are dense in $\HH(b)$.

\item
The function $b$ is in $\HH(b)$, and $\|b\|^2_{\HH(b)}=|a(0)|^{-2}-1$.
\item
The space $\HH(b)$ is invariant by the unilateral shift $S$.

\end{enumerate}
\end{theorem}

\begin{proof}
By checking the action on monomials, it is immediate that the space $\PP_n$ of polynomials of degree less or equal to $n$ is invariant by $T_{\bar a}$. But $T_{\bar a}$ is one-to-one, and so $T_{\bar a}|\PP_n$ is also onto. So all polynomials belong to the image of $T_{\bar a}$, which is $\MM(\bar a)$. Moreover, since $T_{\bar a}$ is one-to-one, it is unitary as an operator from $H^2$ to $\MM(\bar a)$.  Then the image of all polynomials, which form a dense set in $H^2$, is a dense set in $\MM(\bar a)$ which proves (1).

Suppose that $h\in\HH(b)$ is orthogonal in $\HH(b)$ to all $\MM(\bar a)$. In particular, $h$ is orthogonal to $T_{\bar a}S^*{}^n h$ for every $n\ge 0$. By Lemma~\ref{le:H(b) and M(bar a) for b nonextreme}(3), $(T_{\bar a}S^*{}^n h)^+=T_{\bar a}S^*{}^n h^+$; applying then ~\ref{le:H(b) and M(bar a) for b nonextreme}(2), we have, for any $n\ge 0$,
\begin{align*}
0&= \< h,T_{\bar a}S^*n h \>_	{\HH(b)}\\
&=\< h,T_{\bar a}S^*n h \>_{H^2}+\< h^+,T_{\bar a}S^*n h^+\>_{H^2}\\
&=\frac{1}{2\pi} \int_{-\pi}^\pi a(e^{it})(|h(e^{it})|^2+ |h^+(e^{it})|^2) e^{int}\,dt.
\end{align*}
Therefore, the function $a(|h|^2+|h^+|^2)$ belongs to $H^1_0$. A classical fact about outer functions (see, for instance,~\cite{N}) implies that we also have  $|h|^2+|h^+|^2\in H^1_0$. But the only real-valued function in $H^1_0$ is the zero function, so $h=0$, which proves (2). Obviously, (1) and (2) imply (3).

We have
\begin{equation*}
T_{\bar b}b =P_+(\bar b b)=P_+(1-\bar a a)=T_{\bar a}(1/\overline{a(0)}-a). 
\end{equation*}
By Lemma~\ref{le:H(b) and M(bar a) for b nonextreme} it follows that $b\in \HH(b)$ and $b^+=1/\overline{a(0)}-a$; moreover,
\begin{align*}
\|b\|^2_{\HH(b)}&= \|b\|^2_{H^2}+\|1/\overline{a(0)}-a\|^2_{H^2}\\
&=\|b\|^2_{H^2}+ \|a-a(0)\|^2_{H^2}+\|1/\overline{a(0)}-a(0)\|^2_{H^2}\\
&=\|b\|^2_{H^2}+ \|a\|^2_{H^2}-|a(0)|^2 +|a(0)|^{-2}+|a(0)|^2 -2\\
&=|a(0)|^{-2}-1,
\end{align*}
which proves (4).

Finally, Lemma~\ref{le:adjoint of X} together with (4) prove the invariance of $\HH(b)$ to $S$.
\end{proof}

\section{The extreme case}

We point out first some differences with respect to the nonextreme case.

\begin{theorem}\label{th:extreme: b is not there}
Suppose $b$ is extreme. Then:
\begin{enumerate}
\item
The function $b$ does not belong to $\HH(b)$.
\item
If $b\not=\bf 1$, then $\HH(b)$ is not invariant by $S$.
\end{enumerate}
\end{theorem}

\begin{proof}
Suppose $b\in \HH(b)$. By Theorem~\ref{th:geometrical representation of de branges spaces}, it follows that there exists $\psi\in\overline{\Delta H^2}\subset L^2$, such that $b\oplus\psi \perp \{bh\oplus \Delta h: h\in H^2\}$. So 
\[
\< b\oplus \psi,bh \oplus \Delta h \>=0, \text{ for all }h\in H^2,
\]
which is equivalent to $|b|^2+\Delta\psi\in \overline{H^2_0}$. This is equivalent to $1-\Delta^2+\Delta \psi \in \overline{H^2_0}$, whence $f:=\Delta^2-\Delta\bar\psi$ is a nonzero (note that its zeroth Fourier coefficient is~1) function in $H^2$. Thus $\Delta^{-1}f=\Delta-\bar\psi\in L^2$, or $\Delta^{-2}|f|^2\in L^1$.

We assert that this is not possible. Indeed, 
\[
\Delta^{-2}|f|^2\ge \log (\Delta^{-2}|f|^2)=2\log|f|-2\log\Delta.
\]
Integrating, we obtain
\[
\frac{1}{2\pi}\int \Delta^{-2}(e^{it})|f(e^{it})|^2\,dt\ge 2\frac{1}{2\pi}\int\log|f(e^{it})|\,dt+2\frac{1}{2\pi}\int(-\log\Delta(e^{it}))\,dt,
\]
which cannot be true, since the first two integrals are finite, while the third is infinite by Lemma~\ref{le:characterization of extreme}.

We have thus proved (1). Then (2) follows from Lemma~\ref{le:adjoint of X}, which can be restated as
\[
\< h,S^*b_{\HH(b)} \> b=Sh-X_b^*h.
\]
So, if we choose $h$ not orthogonal (in $\HH(b)$) to $S^*b$ (in particular, $h=S^*b$), we must have $Sh\notin \HH(b)$.
\end{proof}

In the sequel we will use the geometrical representation of $\HH(b)$ given by Theorem~\ref{th:geometrical representation of de branges spaces}. Using Lemma~\ref{le:characterization of extreme}, we may replace in its statement $\overline{\Delta H^2}$ by $\overline{\Delta L^2}$.

Denote $\tilde b(z)=\overline{b(\bar z)}$, and $\tilde \Delta=(1-|\tilde b|^2)^{1/2}$. The map $f\mapsto \tilde f$ is a unitary involution that maps $L^2$ onto $L^2$, $H^2$ onto $H^2$, and $\overline{\tilde\Delta L^2}$ onto $\overline{\Delta L^2}$.

\begin{exercise}\label{exr:b and tilde b extreme}
$b$ is extreme if and only if $\tilde{b}$ is extreme.
\end{exercise}

\begin{theorem}\label{th:adjoint for the extreme case}
Suppose $b$ is extreme. 
Define
\begin{equation*}
\Omega: L^2\oplus \overline{\tilde\Delta L^2}\to L^2\oplus \overline{\Delta L^2}
\end{equation*}
by the formula
\begin{equation*}
\Omega (f\oplus g)=\Big(b(z)\bar z f(\bar z)+\Delta \bar z g(\bar z)\Big) \oplus \Big(\Delta \bar z f(\bar z)-\bar b \bar z g(\bar z)\Big).
\end{equation*}
Then $\Omega$ is unitary, it maps $\KK_{\tilde b}$ onto $\KK_b$, and
\begin{equation}\label{eq:Xb*=X_tilde b}
\Omega X_{\tilde b}= X_b^*\Omega.
\end{equation}
\end{theorem}

\begin{proof}
$\Omega$ acts as the unitary $f\oplus g\mapsto \bar z f(\bar z)\oplus  \bar z g(\bar z)$ followed by the unitary $J(M_b)$, so it is unitary. We have 
\[
\begin{split}
\KK_b&=(L^2\oplus \overline{\Delta L^2})\ominus \big[(H^2_-\oplus \{0\})\oplus (\{bf\oplus \Delta f:f\in H^2\})\big],\\
\KK_{\tilde b}&=(L^2\oplus \overline{\tilde \Delta L^2})\ominus \big[(H^2_-\oplus \{0\})\oplus (\{\tilde bf\oplus\tilde \Delta f:f\in H^2\})\big].
\end{split}
\]
If $f\in H^2_-$, then
\[
\Omega (f\oplus 0)=b \bar z f(\bar z)\oplus \Delta \bar z f(\bar z).
\]
But the map $f\mapsto \bar z f(\bar z)$ is a unitary from $H^2_-$ onto $H^2$, whence it follows that $\Omega(H^2_-\oplus \{0\})= \{bf\oplus \Delta f:f\in H^2\}$. Similarly we obtain $\Omega(\{bf\oplus \tilde\Delta f:f\in H^2\})=H^2_-\oplus \{0\}$. Therefore
\[
\Omega((H^2_-\oplus \{0\})\oplus (\{\tilde bf\oplus\tilde \Delta f:f\in H^2\}))=(H^2_-\oplus \{0\})\oplus (\{bf\oplus \Delta f:f\in H^2\}),
\]
whence $\Omega(\KK_{\tilde b})=\Omega (\KK_b)$.

Finally, we have $X_b=P_{\KK_b}(Z^*\oplus Z_\Delta^*)P_{\KK_b}|\KK_b$ and $X_{\tilde b}=P_{\KK_{\tilde b}}(Z^*\oplus Z_{\tilde\Delta}^*)P_{\KK_{\tilde b}}|\KK_{\tilde b}$. But $\Omega(\KK_{\tilde b})=\Omega (\KK_b)$ implies $P_{\KK_{\tilde b}}=\Omega^* P_{\KK_{b}}\Omega$. Therefore
\[
\begin{split}
\Omega X_{\tilde b}&=\Omega P_{\KK_{\tilde b}}(Z^*\oplus Z^*_{\tilde\Delta})P_{\KK_{\tilde b}}|\KK_{\tilde b}=\Omega \Omega^* P_{\KK_{b}}\Omega(Z^*\oplus Z^*_{\tilde\Delta})\Omega^* P_{\KK_{b}}\Omega|\KK_{\tilde b}\\&=P_{\KK_{b}}\Omega(Z^*\oplus Z^*_{\tilde\Delta})\Omega^* P_{\KK_{b}}\Omega|\KK_{\tilde b}.
\end{split}
\]
But one checks easily that $\Omega(Z^*\oplus Z^*_{\tilde\Delta})\Omega^*=Z\oplus Z_\Delta$, and therefore
\[
\begin{split}
P_{\KK_{b}}\Omega(Z^*\oplus Z^*_{\tilde\Delta})\Omega^* P_{\KK_{b}}\Omega|\KK_{\tilde b}&=
P_{\KK_{b}}(Z\oplus Z_\Delta)P_{\KK_{b}}\Omega|\KK_{\tilde b}\\&= (P_{\KK_{b}}(Z^*\oplus Z^*_\Delta)P_{\KK_{b}})^*\Omega|\KK_{\tilde b} = X_b^*\Omega,
\end{split}
\]
which ends the proof of the theorem.
\end{proof} 

We note that  in case $b$ is nonextreme $X_b^*$ is \emph{never} unitarily equivalent to $X_{\tilde b}$.

\begin{theorem}\label{th:properties of X for b extreme}
Suppose $b$ is extreme. Then $\dim \DD_{X_b}=\dim \DD_{X_b^*}=1$, and there exists no subspace of $\HH_b$ invariant by $X$ and such that its restriction therein is an isometry.
\end{theorem}

\begin{proof}
From Theorem~\ref{th:geometrical representation of de branges spaces}(3) it follows that we may prove the properties for the restriction  $S^*\oplus Z^*_\Delta|\KK_b$. Since $S^*$ acts isometrically on $H^2_0$ and $Z^*_\Delta$ is unitary, we have $\DD_{S^*\oplus Z^*_\Delta}=\bbC (1\oplus 0)$. By Exercise~\ref{exr:defect of restriction}, $\DD_{X_b}=\bbC P_{\KK_b}(1\oplus 0)$. But $P_{\KK_b}(1\oplus 0)\not=0$; indeed, otherwise we would have $1\oplus 0=bh\oplus \Delta h$ for some $h\in H^2$; since $h\not=0$ a.e., this would imply $\Delta=0$ a.e., or $b$ inner, which is impossible if $1=bh$. Therefore $\dim\DD_X=1$.

Applying the same argument to $\tilde b$ and using~\eqref{eq:Xb*=X_tilde b}, it follows that  $\dim\DD_{X^*}=1$.

Finally, suppose $\YY\subset\KK$ is a closed subspace on which $X_b$ acts isometrically, and $h\oplus g\in \YY$, we have, for any $n\ge 0$,
\[
\|h\|^2+\|g\|^2=\|h\oplus g\|^2=\|S^*{}^nh\oplus Z_\Delta^*{}^ng\|^2=
\|S^*h{}^n\|^2+\| Z_\Delta^*{}^ng\|^2\to \| Z_\Delta^*{}^ng\|^2=\|g\|^2,
\]
whence $h=0$. But then we must have $0\oplus g\perp bf\oplus \Delta f$ for all $f\in H^2$, or $g\perp \overline{\Delta H^2}=\overline{\Delta L^2}$. Since, on the other hand, $g\in\overline{\Delta L^2}$, it follows that $g=0$, which ends the proof of the theorem.
\end{proof}

\section{$\HH(b)$ as a model space}

The purpose of this section is to prove the converse of Theorem~\ref{th:properties of X for b extreme}. This  will show that in the extreme case the de Branges--Rovnyak spaces are model spaces for a large class of operators.
The theorem below (as well as its proof) is in fact a particular case of the much more general analysis of contractions done in the Sz.Nagy--Foias theory (see~\cite{NF}). Here we have adapted the argument to a ``minimal'' self-contained form.

\begin{theorem}\label{th:model property for b extreme}
Suppose $T\in\BB(H)$ is a contraction such that $\dim \DD_T=\dim \DD_{T^*}=1$, and there exists no subspace of $H$ invariant by $T$ and such that its restriction therein is an isometry. Then there exists an extreme $b$ in the  unit ball of $H^\infty$ such that $T$ is unitarily equivalent to $X_b$.
\end{theorem}

\begin{proof} Since the proof is rather long, we divide it in several steps.

\smallskip
\noindent{\bf Step 1. Dilation of $T$}.
To find the required function $b$, we will develop a certain geometrical construction.
Changing the order of the components in the range of the Julia operator yields a unitary operator mapping $H\oplus \DD_{T^*}$ into $\DD_T\oplus H$ according to the matrix $\left(\begin{smallmatrix} D_T & -T^*\\T &  D_{T^*}\end{smallmatrix}\right)
$. We can extend this unitary to a unitary $W$ acting on the single enlarged space 
\[\HHH=\dots\oplus \DD_{T^*}\oplus \DD_{T^*}\oplus H\oplus \DD_T\oplus \DD_T\oplus \dots,\]
that can be written as an bi-infinite operator matrix:
\begin{equation}\label{eq:the big matrix}
W= \begin{pmatrix}
\ddots & & & &&&&\\
&1&&&&&&\\
&& 1 &&&&&\\
&&& D_{T}&
-T^* &&&\\
&&& \framebox{T}& D_{T^*} &&&\\
&&&&&1&&\\
&&&&&&1&\\
&&&&&&&\ddots
\end{pmatrix}
\end{equation}
where the boxed entry corresponds to the central entry $T:H\to H$. If we write $\HHH=\HHH_-\oplus H\oplus \HHH_+$, with 
\[
\HHH_-=\dots\oplus \DD_{T^*}\oplus \DD_{T^*}, \quad \HHH_+=\DD_T\oplus \DD_T\oplus \dots,
\]
then $\HHH_-$ is invariant by $W$, which acts therein as translation to the left, while $\HHH_+$ is invariant by $W^*$, whose restriction is translation to the right. (This is a consequence of the fact that the $1$ entries in the definition of $W$ are all located  \emph{immediately above} the main diagonal.)

\smallskip
\noindent{\bf Step 2. Two embeddings of $L^2$ into $\HHH$.}
Take a unit vector $\epsilon^-_{-1}$ in the $\DD_{T^*}$ component of $\HHH_-$ which is mostly to the right, and define, for $n\in\bbZ$, $\epsilon^-_{n}=W^{-n-1}\epsilon^-_{-1}$. Since  $\dim\DD_{T^*}=1$, the family $(\epsilon^-_n)_{n\le -1}$ forms an orthonormal basis of $\HHH_-$. Moreover,  the whole family $(\epsilon^-_n)_{n\in\bbZ}$ is an orthonormal set in $\HHH$ (exercise!). 

As $(e^{int})_{n\in\bbZ}$ is an orthonormal basis in $L^2$, we may define $\omega_-:L^2\to\HHH$ to be the unique isometry that satisfies $\omega_-(e^{int})=\epsilon^-_n$ for all $n\in\bbZ$. One checks easily that its image $\omega_-(L^2)$ is a reducing space for $W$, and $\omega_-^*W\omega_-=M_{e^{-it}}$. The orthogonal projection onto $\omega_-(L^2)$ is $\omega_-\omega_-^*$, and it commutes with $W$.

An analogous construction can be made for $\HHH_+$. We obtain an orthonormal set $(\epsilon^+_n)_{n\in\bbZ}$ in $\HHH$, such that $(\epsilon^+_n)_{n\ge0}$ is a basis for $\HHH_+$. Then $\omega_+:L^2\to\HHH$ is the isometry that satisfies $\omega_+(e^{int})=\epsilon^+_n$ for all $n\in\bbZ$; $\omega_+(L^2)$ is also a reducing space for $W$, $\omega_+^*W\omega_+=M_{e^{-it}}$, and $\omega_+\omega_+^*W=W\omega_+\omega_+^*$.

\smallskip
\noindent{\bf Step 3. Finding $b$.}
Consider then the map $\omega_-^*\omega_+:L^2\to L^2$. We have, using the above remarks as well as the equalities $\omega_+^*\omega_+=\omega_-^*\omega_-=I_{L^2}$,
\begin{equation}\label{eq:commutation for omega-*omega+}
\begin{split}
\omega_-^*\omega_+M_{e^{-it}}&= \omega_-^*\omega_+ \omega_+^*W\omega_+
=\omega_-^*W\omega_+ \omega_+^*\omega_+=\omega_-^*W\omega_+ \\
&=(\omega_-^*\omega_-)\omega_-^*W\omega_+=
\omega_-^* W \omega_-\omega_-^*\omega_+= 
M_{e^{-it}} \omega_-^*\omega_+.
\end{split}
\end{equation}
So $\omega_-^*\omega_+$ commutes with $M_{e^{-it}}$; it follows that it commutes also with its inverse $M_{e^{it}}$ (exercise!). We have noticed in the introduction that in this case we must have $\omega_-^*\omega_+=M_{b}$ for some function $b\in L^\infty$, and $\|\omega_-^*\omega_+\|\le 1$ implies $\|b\|_\infty\le 1$.

Now, $b=M_b {\bf 1}= \omega_-^*\omega_+ 1=\omega_-^*\epsilon^+_0$. Since $\epsilon^+_0\in\HHH_+\perp\HHH_-=\omega_-(\overline{H^2_0})$, it follows that $\omega_-^*\epsilon^+_0\in H^2$. Thus $b\in L^\infty\cap H^2=H^\infty$, and we have found our candidate for the  function in the unit ball of $H^\infty$. It remains now to check that it satisfies the required properties. As above, we will denote $\Delta=(1-|b|^2)^{1/2}\in L^\infty$.

\smallskip
\noindent{\bf Step 4. Constructing the unitary equivalence.}
Let us now note that the closed linear span $\omega_+L^2\vee \omega_- L^2$ equals $\HHH$. Indeed, it reduces $W$ and  contains $\HHH_+$ and $\HHH_-$; thus its orthogonal $Y$ has to be a reducing subspace of $W$ contained in $H$ (more precisely, in its embedding in $\HHH$). From~\eqref{eq:the big matrix} it follows then that $W|Y=T|Y$, so $Y$ should be a subspace of $H$ invariant by $T$ and such that the restriction is isometric (even unitary!), which contradicts the hypothesis. Thus $Y=\{0\}$.

We define then a mapping $U:\omega_+L^2\vee \omega_- L^2\to L^2\oplus \overline{\Delta L^2}$ by
\begin{equation*}
U(\omega_+ f_++\omega_-f_-)=(f_-+bf_+)\oplus \Delta f_+.
\end{equation*}
We have
\[
\begin{split}
\|\omega_+ f_++\omega_-f_-\|^2&= \|\omega_+f_+\|^2+\|\omega_-f_-\|^2+ 2\Re \<\omega_+ f_+ ,\omega_- f_- \>\\ 
&=\|f_+\|^2_2+\|f_-\|^2_2+2\Re \<\omega_-^* \omega_+ f_+ ,f_- \>_2
=\|f_+\|^2_2+\|f_-\|^2_2+ 2\Re \< bf_+,f_- \>_2,
\end{split}
\]
and
\[
\begin{split}
\| (f_-+bf_+)\oplus \Delta f_+ \|^2&=\int |f_-+bf_+|^2 +\int \Delta^2 |f_+|^2\\
&=\int |f_-|^2 +|bf_+|^2+2\Re bf_+\bar f_- + \Delta^2 |f_+|^2\\&=
\int |f_-|^2 +|f_+|^2+2\Re bf_+\bar f_-,
\end{split}
\]
whence  $U$ is an isometry, and it is easy to see that the image is dense. It can be extended to a unitary operator, that we will denote by the same letter,
\[
U:\HHH\to L^2\oplus \overline{\Delta L^2}.
\]
The commutation relations satisfied by $\omega_\pm$ imply that $UW=(Z^*\oplus Z_\Delta^*)U$.

\smallskip
\noindent{\bf Step 5. Final checks.}
Now $U(\HHH)=U(\HHH_-)\oplus U(H)\oplus U(\HHH_+)$, and
\[
\begin{split}
U(\HHH_-)&=U(\omega_-(\overline{H^2_0}))=\{f_-\oplus 0:f_-\in \overline{H^2_0}=\overline{H^2_0}\oplus \{0\},\\
U(\HHH_+)&=U(\omega_+(H^2))=\{bf_+\oplus \Delta f_+:f_+\in H^2\},
\end{split}
\]
so
\[
\begin{split}
U(H)&=(L^2\oplus \overline{\Delta L^2})\ominus \Big((\overline{H^2_0}\oplus \{0\})\oplus (\{bf_+\oplus \Delta f_+:f_+\in H^2\}) \Big)\\
&= (H^2\oplus \overline{\Delta L^2})\ominus (\{bf_+\oplus \Delta f_+:f_+\in H^2\}).
\end{split}
\]
As shown by~\eqref{eq:the big matrix}, $T$ can be viewed as the compression of $W$ to $H$, so it is unitarily equivalent through $U$ to the compression of $Z^*\oplus Z_\Delta^*$ to $U(H)$. This last is easily seen to be the restriction of $S^*\oplus Z_\Delta^*$ to $(H^2\oplus \overline{\Delta L^2})\ominus (\{bf_+\oplus \Delta f_+:f_+\in H^2\})$.

We are very close to the end: Theorem~\ref{th:geometrical representation of de branges spaces} would end the proof, provided we could replace in the formula for $U(H)$ the space $H^2\oplus \overline{\Delta L^2}$ with $H^2\oplus \overline{\Delta H^2}$ and thus $Z_\Delta$ with $V_\Delta$. But the space 
\[
Y:=(H^2\oplus \overline{\Delta L^2})\ominus(H^2\oplus \overline{\Delta H^2})=\{0\}\oplus (\overline{\Delta L^2}\ominus\overline{\Delta H^2})
\]
is invariant with respect to $S^*\oplus Z_2^*$, which acts on it isometrically. By assumption, we must have $Y=\{0\}$, which means that $H^2\oplus \overline{\Delta L^2}=H^2\oplus \overline{\Delta H^2}$; this finishes the proof.\end{proof}
 
 \begin{exercise}\label{exr:general model implies special model}
 Show that Theorem~\ref{th:model property for b extreme} implies Theorem~\ref{th:unitary equivalence for inner function}.
 \end{exercise}

 \section{Further reading}

We discuss in this section some directions in which the study of de Branges--Rovnyak spaces has developed.

The model spaces $\K_u$ have no nonconstant multipliers. However, the theory of multipliers is interesting for the case of de Branges--Rovnyak spaces corresponding to nonextreme $b$; see references~\cite{LS1, LS2, LS3}.

Integral representations of de Branges--Rovnyak spaces appear in~\cite{BK}, and have further been developed in~\cite{P, FM1, FM2}. The last paper is used in~\cite{BFM} to obtain weighted norm inequalities for functions in  de Branges--Rovnyak spaces.

The connection between de Branges--Rovnyak  and Dirichlet spaces is exploited in~\cite{CGR, CR}; that between de Branges--Rovnyak spaces and composition operators in~\cite{J}.

Finally, it is natural from many points of view (including that of model spaces) to consider also matrix or operator valued de Branges--Rovnyak spaces. These have already been introduced in~\cite{BK}; see~\cite{BBH0, BBH} for some recent developments.

\section*{Acknowledgements}

This work  was partially supported by a grant of the Romanian National Authority for Scientific
Research, CNCS Ð UEFISCDI, project number PN-II-ID-PCE-2011-3-0119.

This work will appear in the Proceedings Volume of the CRM Conference \emph{Invariant Subspaces of the Shift Operator}, August 26--30, 2013. The author is grateful to the organizers of the Conference for the support provided.

\end{document}